\documentclass[11pt]{article}
\usepackage[T1]{fontenc}
\usepackage[utf8]{inputenc}
\usepackage{geometry}
\usepackage{amsmath}
\usepackage{amsthm}
\usepackage{enumitem}
\usepackage{setspace}
\usepackage[unicode=true,
 bookmarks=false,
 breaklinks=false,pdfborder={0 0 1},backref=false,colorlinks=false]
 {hyperref}
\usepackage[dvipsnames]{xcolor}
\allowdisplaybreaks

\makeatletter

\newtheorem{thm}{Theorem}
\newtheorem{lem}{Lemma}
\newtheorem{claim}{Claim}

\newtheorem{innercustomgeneric}{\customgenericname}
\providecommand{\customgenericname}{}
\newcommand{\newcustomtheorem}[2]{%
  \newenvironment{#1}[1]
  {%
   \renewcommand\customgenericname{#2}%
   \renewcommand\theinnercustomgeneric{##1}%
   \innercustomgeneric
  }
  {\endinnercustomgeneric}
}

\newcustomtheorem{customthm}{Theorem}
\newcustomtheorem{customlemma}{Lemma}

\usepackage{amssymb}

\makeatother

\setenumerate[0]{label=(\roman*), ref=\roman*}

\begin{document}
\title{On nearly consecutive sequences without long arithmetic progressions}
\author{Jacob Fox\thanks{Department of Mathematics, Stanford University, Stanford, CA 94305. Email: {\tt jacobfox@stanford.edu}. Research
supported by NSF awards DMS-2154129 and DMS-2452737.} 
\and Carl Schildkraut\thanks{Department of Mathematics, Stanford University, Stanford, CA 94305. Email: {\tt carlsch@stanford.edu}. Research
supported by National Science Foundation Graduate Research Fellowship Program under Grant No. DGE-2146755. } }
\date{}
\maketitle

\begin{abstract}
A sequence $a_1,\ldots,a_n$ of integers is \textit{nearly consecutive} if $a_{i+1}-a_i \in \{1,2\}$ for $1 \leq i \leq n-1$. We prove that there are nearly consecutive sequences  of length $\Omega(2^k/k^2)$ that contain no $k$-term arithmetic progression. This improves on the previous best known bound of Alon and Zaks from 1998. We also prove a generalization for sequences with bounded gaps. 
\end{abstract}

\section{Introduction}

We use $[m,n]$ to denote the set of positive integers $a$ with $m \leq a \leq n$. We let $[n]=[1,n]$. 

A $k$-term arithmetic progression ($k$-AP for short) is a sequence of $k$ numbers with the same difference between consecutive terms. 
A sequence $(a_i)_{i=1}^n$ of integers has the {\it $r$-gap property} if $a_{i+1}-a_i \in [r]$ for all $i \in [n-1]$. Of particular interest is the case $r=2$, and sequences with the $2$-gap property are sometimes referred to in the literature as {\it nearly consecutive sequences}. The {\it gap number} $g(k;r)$ is the minimum integer $n$ such that every sequence of integers of length $n$ with the $r$-gap property contains a $k$-AP.

The van der Waerden number $w(k;r)$ is the minimum $N$ such that every $r$-coloring of $[N]$ contains a monochromatic $k$-AP. That these numbers exist is the content of van der Waerden's theorem. Rabung \cite{Rabung} proved that the existence of the gap numbers is equivalent to van der Waerden’s theorem. Another closely related function is $m(k;r)$, which is the minimum $n$ such that every sequence $(a_i)_{i=1}^n$ with $a_i \in [(i-1)r+1,ir]$ for $i \in [n]$
contains a $k$-AP. Nathanson \cite{Nathanson} proved several inequalities showing a close relationship between these numbers. In particular, he proved that $m(k;r) \leq g(k;2r-1)$ and 
$$g(k;r)/r \leq m(k;r) \leq w(k;r) \leq m((k-1)r+1,r).$$
Thus, getting a good lower bound on $g(k;r)$ is at least as hard as getting a similar lower bound on $w(k;r)$. 

The first exponential lower bound on $g(k;2)$ was proved by Brown and Hare \cite{BH}. This was improved further to $g(k;r) >r^{k-c_r\sqrt{k}}$ by Alon and Zaks \cite{AZ} in 1998. We give the first improvement since. 

\begin{thm}\label{thm1}
Fix an integer $r \geq 2$. For every $k \geq 2$, there is a sequence of integers with the $r$-gap property of length $\Omega_r(r^k/k^2)$ without a $k$-AP. That is, $g(k;r) = \Omega_r(r^k/k^2)$. 
\end{thm}

Hunter and the first author \cite{FH} recently showed that $w(k;3)$ grows faster than exponential in $k$. Hence, $g(k;5)$ grows super-exponential in $k$ too. Theorem \ref{thm1} gives the best known lower bound on $g(k;r)$ for $r \leq 4$. 

\section{Proof of Theorem \ref{thm1}}

As in Alon and Zaks \cite{AZ}, we give a probabilistic construction of a long sequence with the $r$-gap property, and we use the Lov\'asz local lemma to prove that there is a positive probability that it is $k$-AP-free. 
The book of Alon and Spencer \cite{AlSp} is an excellent reference for this lemma and its applications.   

\begin{lem}[Lov\'asz local lemma]
    Let $B_1,\dots,B_n$ be events in a probability space. Suppose that for each $i \in \{1,\dots,n\}$ there is a set $\Gamma(i) \subseteq \{1,\dots,n\}\setminus\{i\}$
such that $B_i$ is mutually independent of the events
$\{B_j : j \notin \Gamma(i) \cup \{i\}\}$. If there exist real numbers $x_1,\dots,x_n \in [0,1)$ such that
\[
\mathbb{P}(B_i) \le x_i \prod_{j \in \Gamma(i)} (1-x_j)
\qquad \text{for all } i,
\]
then
\[
\mathbb{P}\left(\bigcap_{i=1}^n \overline{B_i}\right)
\ge
\prod_{i=1}^n (1-x_i)
> 0.
\]

In particular, with positive probability none of the events $B_i$ occurs.
\end{lem}

\begin{proof}[Proof of Theorem \ref{thm1}]
We may assume that $k$ is sufficiently large as a function of $r$. We will show that there is a sequence $A$ of length at least $n_0=e^{-11r}r^{k}/k^2$ with the $r$-gap property and no $k$-AP. We first describe how to (randomly) construct the long sequence $A$ of integers with the $r$-gap property, and then show that it is $k$-AP-free with positive probability. 

Let $m$ be the largest multiple of $r$ which is at most $(k-1)/6$. As $k$ is sufficiently large as a function of $r$, then $m$ is as well. 
Let $N$ be the least multiple of $4m$ which is at least $2rn_0$, and let $t=N/(4m)$. Partition $[N]$ into $2t$ intervals $I_i=[(i-1)2m+1,i2m]$ each of length $2m$, and let $J_i=[(i-1)2m+1,i2m-m]$ consist of the first $m$ elements of $I_i$ and $K_i=[i2m-m+1,i2m]$ consist of the last $m$ elements of $I_i$.  
Independently for each $i \in [t]$, pick a uniformly random pair $(j_{2i-1},j_{2i}) \in J_{2i-1} \times J_{2i}$ with $j_{2i} \not \equiv j_{2i-1} \pmod r$. Let $A'=(j_i)_{i=1}^{2t}$. The sequence $A'$ is a subsequence of $A$, which we define to be the strictly increasing sequence whose remaining terms are those of the form $a \in (j_i,j_{i+1})$ with $i \in [2t-1]$ and $a \equiv j_i \pmod r$. That is, the elements of the sequence $A$ are $$\{j_1,\ldots,j_{2t}\} \cup \bigcup_{i=1}^{2t-1}\{a \in \mathbb{Z}:j_i<a<j_{i+1},a \equiv j_i \pmod r\}.$$ In particular, for $i \in [2t-1]$, the elements of $A \cap K_i$ are the elements of $K_i$ that are congruent to $j_i \pmod r$. 

We next make some simple observations about the random sequence $A$. Note that the first term of the sequence $A$ is $j_1$ and the last term is $j_{2t}$. The sequence $A$ has the $r$-gap property by construction. Indeed, the consecutive gaps are exactly $r$, apart possibly from those between $j_i$ and the preceding term of $A$, which are each at most $r$. For each $1 < i < 2t$, the sequence $A$ contains at least $2m/r$ elements in interval $I_i$. Indeed, $m$ is a multiple of $r$ and $A$ contains at least one element in each block of $r$ consecutive integers in $I_i$. Hence, $A$ has length at least $2(t-1)2m/r = N/r -4m/r \geq 2n_0-4m/r \geq n_0$. The last inequality uses that $k$ is sufficiently large as a function of $r$. 

To complete the proof, we show that with positive probability, $A$ is $k$-AP-free. The proof uses the Lov\'asz local lemma, and we first make some observations on the dependencies between elements appearing in $A$ and deduce results on the dependencies between the $k$-APs.

Note that $A$ is determined by $A'$. For $a \in [N]$, let $E_a$ be the event that $a \in A$. 

\begin{claim}
Let $a\in [N]$.
\begin{enumerate}
    \item The event $E_a$ is mutually independent of all $E_b$ with $\lvert b-a\rvert \geq 8m$.
    
    \item We have $\mathbb{P}(E_a)\leq 1/r+1/m$.
\end{enumerate}
\end{claim}
\begin{proof}
    We first prove (i). 
    The event $E_a$ for $a \in I_1$ is determined by $j_1$, and the event $E_a$ for $a \in I_i$ with $i>1$ is determined by $j_{i-1}$ and $j_i$. Also, $j_i$ is mutually independent of all other $j_h$ unless $i$ is even and $h=i-1$ or unless $i$ is odd and $h=i+1$. It follows that for $a \in I_i$, the event $E_a$ is mutually independent of all $E_b$ apart from those $b \in I_j$ with $i-2 \leq j \leq i+3$. Property (i) follows.
    
    We now prove (ii). 
    Suppose $a\in I_i$. 
    Partition the interval $I_i$ into $2m/r$ intervals of length $r$, and let $F_a$ be the event that $j_i$ lies in the same such interval of length $r$ as $a$.
    If $a\in J_i$, then $F_a$ occurs with probability $1/(m/r)$, while if $a\in K_i$ then $F_a$ never occurs.
    Conditioned on $F_a$, we have that $a\in A$ implies, if $i=1$, that  
    $a\equiv j_1\pmod r$, and, if $i>1$, that $a\equiv j_i\pmod r$ or $a\equiv j_{i-1}\pmod r$. Therefore, $\mathbb P[E_a \cap F_a]\leq (2/r)P[F_a]$. Conditioned on $\overline{F_a}$, we have $a\in A$ only if $i \geq 2$, $j_i>a$ and $a\equiv j_{i-1}\pmod r$, or if $j_i<a$ and $a\equiv j_i\pmod r$. Either way, we obtain $\mathbb P[E_a\cap\overline{F_a}] \leq (1/r)P[F_a]$. We conclude
    \[\mathbb P[E_a] =\mathbb P[E_a \cap F_a]+P[E_a\cap\overline{F_a}]\leq \frac{2}{r}\mathbb P[F_a]+\frac{1}{r}\mathbb P[\overline{F_a}]\leq \frac 2r \cdot \frac rm+\frac1r \cdot\left(1-\frac rm\right)=\frac1r+\frac1m.\qedhere\]
\end{proof}

\noindent {\bf Probabilities.} For each $k$-AP $P$ with elements in $[N]$, we let $B_P$ be the bad event that $P$ is a subsequence of $A$. We are done if we can show that with positive probability, none of the bad events $B_P$ occur. 

\begin{claim}
    Let $P$ be a $k$-AP in $[N]$. 
    Set $p:=e^{rk/m}r^{-k}$.
    \begin{enumerate}
        \setcounter{enumi}{2}
        \item If $P$ has common difference at least  $8m$, then $B_P$ occurs with probability at most $p$.
    
        \item If $P$ has common difference in the interval $(m/2,8m)$, then $B_P$ occurs with probability at most $p^{1/16}$.
        
        \item If $P$ has common difference at most $m/2$, then $B_P$ never occurs.
    \end{enumerate}
\end{claim}
\begin{proof}
    If $P$ has common difference at least $8m$, then (i) implies that the events $E_a$ for $a\in P$ are mutually independent.
    Therefore, by (ii),
    \[\mathbb{P}[P~\textrm{is a subsequence of}~A] = \prod_{a\in P}\mathbb{P}(E_a)\leq  \left(\frac{1}{r}+\frac{1}{m}\right)^{k} \leq e^{rk/m}r^{-k}=p.\]
    This proves (iii).
    To prove (iv), note that, if $P$ has common difference strictly between $m/2$ and $8m$, then, taking every sixteenth element of $P$, we form an arithmetic progression with length at least $k/16$ and common difference greater than $8m$, and hence these elements appear independently of each other. Hence, the probability that $P$ is a subsequence of $A$ is at most $p^{1/16}$.

    To prove (v), suppose that the first element of $P$ is in interval $I_i$. The last element of $P$ is at least $k-1$ larger than the first, and since $k-1 \geq 6m$, the last element of $P$ is in an interval $I_j$ with $j \geq i+3$. Since the common difference of $P$ is at most $m/2$ and $K_{i+1}$ has length $m$, there are two consecutive elements of $P$ which lie in $K_{i+1}$. For $B_P$ to possibly occur, these consecutive elements are congruent modulo $r$, and since their difference is a multiple of $r$, all elements of $P$ must be congruent modulo $r$. However, letting $h$ be the smallest odd integer with $h \geq i$ (so $h=i$ or $i+1$), the elements of $A$ in $K_h$ and $K_{h+1}$ are not congruent modulo $r$, and there is an element of $P$ in $K_h$ and another in $K_{h+1}$, which implies $P$ cannot be a subsequence of $A$.  
\end{proof}

\vspace{2mm}

\noindent {\bf Dependencies.} Observe that $B_P= \bigcap_{a \in P} E_a$ is the intersection of the $k$ events $E_a$ with $a \in P$. For a $k$-AP $P$, let $\Gamma(P)=\{P' \not = P:~\textrm{there exists}~a \in P,b \in P'~\textrm{with}~|a-b|<8m\}$. Since the random pairs $(j_{2q-1},j_{2q})$ are mutually independent, the definition of $\Gamma(P)$ implies that $B_P$ is mutually independent of all $B_{P'}$ with $P' \not \in \Gamma(P) \cup \{P\}$.

The number of $k$-APs that lie in $[N]$ that contain any particular element of $[N]$ is at most $2N$. Indeed, for $a \in [N]$, a $k$-AP containing $a$ is determined by $a$'s position (there are at most $k$ positions possible) and the common difference (the common difference is less than $N/(k-1)$), and $k \cdot N/(k-1) \leq 2N$.
For each $a$, there are at most $16m$ elements $b\in[N]$ with $|a-b|<8m$. As $P$ has $k$ elements, each in at most $2N$ arithmetic progressions of length $k$, we obtain $|\Gamma(P)| \leq (16m)k(2N)=32mkN$.

For each $b\in[N]$, the number of $k$-APs $P'$ with common difference at most $8m$ which contain $b$ is at most $8mk$. So, by (i), each $E_a$ is mutually independent of all but $8mk\cdot 16m=128m^2k$ such $B_{P'}$.
We conclude that, for each $k$-AP $P$, there are at most $128m^2k^2$ progressions $P' \in \Gamma(P)$ with common difference at most $8m$.

\vspace{2mm}

\noindent {\bf Choice of parameters for the Lov\'asz local lemma.} Let $x_1=e^2p$ and $x_2=e^2p^{1/16}$. It follows from (\ref{ineq1}) and (\ref{ineq2}) that $x_1,x_2 \leq 1/2$. For a $k$-AP $P$ with common difference at least $8m$, let $x_P= x_1$. For a $k$-AP $P$ with common difference strictly between $m/2$ and $8m$, let $x_P= x_2$. For a $k$-AP $P$ with common difference at most $m/2$, let $x_P=0$.

\vspace{2mm}

\noindent {\bf Checking the Conditions.} Finally, we check that the inequalities are satisfied to apply the Lov\'asz local lemma. We first check two inequalities. From how we chose $m$, it follows that $6 \leq k/m \leq 7$, and we have
\begin{equation}\label{ineq1} 64mkNx_1=64e^2mkNe^{rk/m}r^{-k} \leq 11e^{2}k^2Ne^{rk/m}r^{-k} \leq 11e^{2}k^2Ne^{7r}r^{-k} \leq 33re^{2-4r}\leq 1.\end{equation}
The second to last inequality follows from substituting in an appropriate upper bound on $N$. Indeed, since $N$ is the least multiple of $4m$ that is at least $2rn_0$, we have $N \leq 4m+2rn_0$. As $k$ is large, $4m \leq rn_0$, and hence $N \leq 3rn_0 =3re^{-11r}r^{k}/k^2$.

As $k$ is sufficiently large in $r$, we also have \begin{equation}\label{ineq2} 256m^2k^2 x_2 = 256e^2m^2k^2p^{1/16} \leq 64 k^4e^{7r/16}r^{-k/16} \leq 1.\end{equation}

In addition to inequalities (\ref{ineq1}) and (\ref{ineq2}), we will also use the inequality $1-x \geq e^{-2x}$ for $0 \leq x \leq 1/2$. For $P$ a $k$-AP with common difference greater than $m/2$, we have
\[x_P\prod_{P' \in \Gamma(P)}(1-x_{P'}) \geq x_P(1-x_1)^{32mkN}(1-x_2)^{128m^2k^2} \geq x_Pe^{-64mkNx_1}e^{-256m^2k^2x_2} \geq x_Pe^{-2} \geq \mathbb{P}[B_P].\]

The first inequality follows from splitting the product over $P' \in \Gamma(P)$ depending on whether or not the common difference of $P'$ is at least $8m$.  The last inequality uses (iii) and (iv). We also have from (v) that $x_P=\mathbb{P}[B_P]=0$ if $P$ has common difference at most $m/2$.

Hence, by the Lov\'asz local lemma, with positive probability, $A$ is $k$-AP-free, completing the proof.
\end{proof}

\noindent {\bf AI usage:} The authors derived the mathematics and wrote a complete draft of the paper. ChatGPT was then used for helpful suggestions in editing the paper.

\end{document}